\documentclass[reqno]{amsart}
\usepackage{amsmath,amssymb}
\usepackage{hyperref}

\newcommand*{\mailto}[1]{\href{mailto:#1}{\nolinkurl{#1}}}

\newtheorem{theorem}{Theorem}[section]
\newtheorem{proposition}[theorem]{Proposition}

\newtheorem{lemma}[theorem]{Lemma}

\theoremstyle{definition}

\newtheorem{remark}[theorem]{Remark}

\numberwithin{equation}{section}

\begin{document}

\title[Perturbations of periodic Sturm--Liouville operators]{Perturbations of periodic Sturm--Liouville operators}

\author[J. Behrndt]{Jussi Behrndt}
\address{Technische Universit\"{a}t Graz\\
Institut f\"ur Angewandte Mathematik\\
Steyrergasse 30\\
8010 Graz, Austria}
\email{\mailto{behrndt@tugraz.at}}
\urladdr{\url{https://www.math.tugraz.at/~behrndt/}}

\author[P. Schmitz]{Philipp Schmitz}
\address{Department of Mathematics\\
	Technische Universit\"at Ilmenau\\ Postfach 100565\\
	98648 Ilmenau\\ Germany}
\email{\mailto{philipp.schmitz@tu-ilmenau.de}}
\urladdr{\url{https://www.tu-ilmenau.de/universitaet/fakultaeten/fakultaet-mathematik-und-naturwissenschaften/profil/institute-und-fachgebiete/institut-fuer-mathematik/profil/fachgebiet-angewandte-funktionalanalysis/team}}

\author[G. Teschl]{Gerald Teschl}
\address{Faculty of Mathematics\\
Nordbergstrasse 15\\ 1090 Wien\\ Austria\\ and
International Erwin Schr\"odinger Institute for Mathematical Physics\\
Boltzmanngasse 9\\ 1090 Wien\\ Austria}
\email{\mailto{Gerald.Teschl@univie.ac.at}}
\urladdr{\url{https://www.mat.univie.ac.at/~gerald/}}

\author[C. Trunk]{Carsten Trunk}
\address{Department of Mathematics\\
Technische Universit\"at Ilmenau\\ Postfach 100565\\
98648 Ilmenau\\ Germany}
\email{\mailto{carsten.trunk@tu-ilmenau.de}}
\urladdr{\url{https://www.tu-ilmenau.de/universitaet/fakultaeten/fakultaet-mathematik-und-naturwissenschaften/profil/institute-und-fachgebiete/institut-fuer-mathematik/profil/fachgebiet-angewandte-funktionalanalysis}}


\begin{abstract} We study perturbations of the self-adjoint periodic Sturm--Liouville operator
\[
	A_0 = \frac{1}{r_0}\left(-\frac{\mathrm d}{\mathrm dx} p_0 \frac{\mathrm d}{\mathrm dx} + q_0\right)
\]
and conclude under $L^1$-assumptions on the differences of the coefficients that the essential spectrum and absolutely continuous spectrum remain the same.
If a finite first moment condition holds for the differences of the coefficients, then at most finitely many eigenvalues appear in the spectral
gaps. This observation extends a seminal result by Rofe-Beketov from the 1960s. Finally, imposing a second moment condition we show that
the band edges are no eigenvalues of the perturbed operator.
\end{abstract}

\maketitle

\section{Introduction}

Consider a periodic Sturm--Liouville
differential expression of the form
\begin{equation*}
\tau_0=\frac{1}{r_0}\left(-\frac{\mathrm d}{\mathrm dx}p_0\frac{\mathrm d}{\mathrm dx}+q_0\right)
\end{equation*}
on $\mathbb R$, where  $1/p_0,q_0,r_0 \in L^1_{\mathrm{loc}}(\mathbb R)$ are real-valued and $\omega$-periodic, and $r_0> 0$, $p_0>0$ a.\,e.
Let $A_0$ be the corresponding self-adjoint operator in the weighted $L^2$-Hilbert space $L^2(\mathbb R;  r_0)$ and recall that
the spectrum of $A_0$ is semibounded from below,
purely absolutely continuous and consists of
(finitely or infinitely many) spectral bands;
cf.\ \cite{BrownEasthamSchmidt13}, \cite{K16} or \cite[Section 12]{Weidmann87}.

Now let $1/p_1,q_1,r_1 \in L^1_{\mathrm{loc}}(\mathbb R)$ be real-valued with $r_1> 0$, $p_1>0$ a.\,e., assume that the condition
\begin{align}
\label{xx1}
\int_{\mathbb R} \left(\lvert r_1(t)-r_0(t)\rvert + \left\lvert\frac{1}{p_1(t)}-\frac{1}{p_0(t)}\right\rvert + \lvert q_1(t)-q_0(t)\rvert \right)\lvert t\rvert^k\,\mathrm dt <\infty
\end{align}
holds for $k=0$ in a first step, and consider the corresponding perturbed Sturm--Liouville
differential expression
\begin{equation*}
\tau_1=\frac{1}{r_1}\left(-\frac{\mathrm d}{\mathrm dx}p_1\frac{\mathrm d}{\mathrm dx}+q_1\right)
\end{equation*}
on $\mathbb R$. It turns out that $\tau_1$ is in the limit point case at both singular endpoints $\pm\infty$ and hence there is a unique self-adjoint realization $A_1$
of $\tau_1$ in the weighted $L^2$-Hilbert space $L^2(\mathbb R;  r_1)$. The first observation in Theorem~\ref{thm1} below is that
the essential spectra of $A_0$ and $A_1$ coincide and the interior is purely absolutely continuous spectrum of $A_1$.
In the special case $r_0=r_1=p_0=p_1=1$ this result is known from \cite{S91} and for $p_0\not=p_1$ a related result is contained in
\cite{BrownEasthamSchmidt13}; cf. Remark~\ref{jussirem}.

\begin{theorem}\label{thm1}
Assume that condition \eqref{xx1} holds for $k=0$ and let
$A_0$ and $A_1$ be the self-adjoint realizations of $\tau_0$ and $\tau_1$ in $L^2(\mathbb R;  r_0)$ and $L^2(\mathbb R;  r_1)$, respectively. Then
we have
\begin{equation*}
\sigma_{\mathrm{ess}}(A_0) = \sigma_{\mathrm{ess}}(A_1),
\end{equation*}
the spectrum of $A_1$ is purely absolutely continuous in
the interior of the spectral bands,
and
$A_1$ is semibounded from below.
\end{theorem}

In particular, the band structure of the spectrum of the periodic operator $A_0$ is preserved for the essential spectrum of $A_1$ and in the gaps
of $\sigma_{\mathrm{ess}}(A_1)$ discrete eigenvalues may appear that may accumulate to the edges of the spectral bands; for a
detailed discussion in the case $r_0=r_1$ we refer to \cite[Section 5.3]{BrownEasthamSchmidt13}.
Our second main objective in this note is to verify that under a finite first moment condition on the difference of the coefficients
there are at most finitely many discrete eigenvalues in the gaps of the
essential spectrum of $A_1$. The question whether eigenvalues accumulate at the band edges has a long
tradition going back to the seminal results of Rofe-Beketov \cite{ROFE-BEKETOV}, which were later extended by Schmidt \cite{S00} (see
also \cite[\S 5.4]{BrownEasthamSchmidt13} for the special case $r_0=r_1=1$ and $p_0=p_1$). They play also an important role for the
scattering theory in this setting \cite{F1,F2,F3, G13}.
The currently best results in this direction can be found in \cite{kt3}, which apply in the special case $r_0=r_1$.

\begin{theorem}\label{thm2}
Assume that condition \eqref{xx1} holds for $k=1$ and let
$A_0$ and $A_1$ be the self-adjoint realizations of $\tau_0$ and $\tau_1$ in $L^2(\mathbb R;  r_0)$ and $L^2(\mathbb R;  r_1)$, respectively.
Then every gap of the spectral bands $\sigma_{\mathrm{ess}}(A_0) =\sigma_{\mathrm{ess}}(A_1)$ contains at most finitely many eigenvalues of $A_1$.
\end{theorem}

In the third result we pay special attention to the edges of the spectral bands. If \eqref{xx1} holds for $k=1$ (and hence also for $k=0$), then the interior of the
spectral bands consists of purely absolutely continuous spectrum of $A_1$ and the eigenvalues of $A_1$ in the gaps do not accumulate
to the band edges. If we further strengthen the assumptions and impose a finite second moment condition $k=2$ in \eqref{xx1} (and hence also $k=1$ and $k=0$),
then it turns out that the band edges are no eigenvalues of $A_1$.

\begin{theorem}\label{thm3}
Assume that condition \eqref{xx1} holds for $k=2$ and let
$A_0$ and $A_1$ be the self-adjoint realizations of $\tau_0$ and $\tau_1$ in $L^2(\mathbb R;  r_0)$ and $L^2(\mathbb R;  r_1)$, respectively.
Then the edges of the spectral bands $\sigma_{\mathrm{ess}}(A_0) =\sigma_{\mathrm{ess}}(A_1)$ are no eigenvalues of $A_1$ and the spectral bands
consist of purely absolutely continuous spectrum of $A_1$.
\end{theorem}

In Section~\ref{sec2} we also show that the claim in Theorem~\ref{thm3} remains valid if \eqref{xx1} holds for $k=1$ and some other additional
assumptions for $r_1$ and $q_1$ are satisfied; cf. Proposition~\ref{propi}.
Our proofs of Theorems~\ref{thm1}--\ref{thm3} are based on a careful analysis of the solutions of $(\tau_0-\lambda)u=0$ and
$(\tau_1-\lambda)u=0$ for $\lambda\in\mathbb R$; cf.\ Lemma~\ref{pingpong} and Lemma~\ref{lachs}. While the properties of the solutions of the periodic
problem in Lemma~\ref{pingpong} are mainly consequences of well-known properties of the Hill discriminant, the properties of the solutions of the perturbed
problem in Lemma~\ref{lachs} require some slightly more technical arguments.
It is convenient to first verify variants of Theorems~\ref{thm1}--\ref{thm3}
for self-adjoint realizations of $\tau_0$ and $\tau_1$ on half-lines
$(-\infty,a)$ and $(a,\infty)$ with finite endpoint $a$, and use a coupling argument to conclude the corresponding results on $\mathbb R$. One
of the key ingredients is the connection of the zeros of a modified Wronskian with the finiteness of the spectrum from \cite{GesztesySimonTeschl96}.

\vskip 0.1cm
\noindent {\bf Acknowledgments.}
We are indebted to Fritz Gesztesy for very helpful discussions and literature hints.

\section{Perturbations of periodic Sturm--Liouville operators on a half-line}\label{sec2}

We prove
variants of Theorems~\ref{thm1}--\ref{thm3} for self-adjoint realizations $H_0$ and $H_1$ of $\tau_0$ and $\tau_1$, respectively,
in the $L^2$-spaces $L^2((a,\infty);r_0)$ and $L^2((a,\infty);r_1)$ with some finite endpoint $a$. For the real-valued coefficients we have
$1/p_j,q_j,r_j \in L^1_{\mathrm{loc}}([a,\infty))$ and $r_j> 0$, $p_j>0$ a.\,e., and $1/p_0,q_0,r_0$ are $\omega$-periodic.

The differential expression $\tau_0$ is in the limit point case at $\infty$ and regular at $a$. In the following
let $H_0$ be any self-adjoint realization of $\tau_0$ in $L^2((a,\infty);r_0)$.
Similar as in the full line case also
on the half-line the essential spectrum of $H_0$ is purely absolutely continuous and  consists of infinitely many closed intervals
\begin{equation}\label{harry}
\sigma_{\mathrm{ess}}(H_0) = \bigcup_{k=1}^\infty [\lambda_{2k-1},\lambda_{2k}],
\end{equation}
where the endpoints $\lambda_{2k-1}$ and $\lambda_{2k}$, $\lambda_{2k-1}<\lambda_{2k}$, denote the $k$-th eigenvalues of the regular Sturm--Liouville operator in $L^2((a,a+\omega);r_0)$ (in nondecreasing order) with
periodic and semiperiodic boundary conditions, respectively;
cf.\ \cite{BrownEasthamSchmidt13} or \cite[Section 12]{Weidmann87} for more details. Recall that the closed intervals may adjoin and that also
$\sigma_{\mathrm{ess}}(H_0)=[\lambda_1,\infty)$ may happen in \eqref{harry}.
Each interval $(-\infty,\lambda_1)$ and $(\lambda_{2k}, \lambda_{2k+1})$, $k\in\mathbb N$, may contain at most one (simple) eigenvalue of $H_0$.
In particular, $H_0$ is semibounded from below and \eqref{harry} implies that the interior of $\sigma_{\mathrm{ess}}(H_0)$ is non-empty.

\begin{theorem}\label{thm1a}
Assume that
\begin{align}\label{xx1a}
\int_{a}^\infty \left(\lvert r_1(t)-r_0(t)\rvert + \left\lvert\frac{1}{p_1(t)}-\frac{1}{p_0(t)}\right\rvert + \lvert q_1(t)-q_0(t)\rvert \right)\,\mathrm dt <\infty
\end{align}
and let
$H_0$ and $H_1$ be arbitrary self-adjoint realizations of $\tau_0$ and $\tau_1$ in $L^2((a,\infty);  r_0)$ and $L^2((a,\infty);  r_1)$, respectively. Then
we have
\begin{equation*}
\sigma_{\mathrm{ess}}(H_0) = \sigma_{\mathrm{ess}}(H_1),
\end{equation*}
the spectrum of $H_1$ is purely absolutely continuous in
the interior of the spectral bands,
and
$H_1$ is semibounded from below.
\end{theorem}
It follows  that $H_1$ has non-empty essential spectrum, hence,
the differential expression $\tau_1$ is in the limit point case at $\infty$.

\begin{remark}\label{jussirem}
For the special case $r_0=r_1=p_0=p_1=1$ the result in
Theorem~\ref{thm1a} goes back to the paper \cite{S91} of G.\
Stolz, where
instead of the assumption $q_1-q_0\in L^1(a,\infty)$ in \eqref{xx1a} the
weaker conditions
\begin{equation}\label{con1}
\int_c^\infty \vert (q_1-q_0)(t+\omega)-(q_1-q_0)(t)\vert \,\mathrm dt
<\infty
\end{equation}
for some $c>a$ and
\begin{equation}\label{con2}
  \lim_{x\rightarrow\infty}\int_x^{x+1}\vert q_1(t)-q_0(t)\vert\,\mathrm
dt=0
\end{equation}
are imposed. The considerations from \cite{S91} are extended in
\cite[Chapter 5.2]{BrownEasthamSchmidt13} to the case $r_0=r_1$ and
$p_0\not=p_1$ with
$1/p_1 - 1/p_0$ satisfying similar conditions
\eqref{con1}--\eqref{con2}. More precisely, in \cite[Corollary
5.2.3]{BrownEasthamSchmidt13} it was shown that the interior
of the essential spectrum of $H_0$ is purely absolutely continuous
spectrum of $H_1$ and hence $\sigma_{\mathrm{ess}}(H_0) \subset
\sigma_{\mathrm{ess}}(H_1)$.
For the other inclusion in \cite[Theorem 5.3.1]{BrownEasthamSchmidt13} it is assumed that
$r_0=r_1$, $p_0=p_1$ together with additional limit conditions for $q_1-q_0$.
For details we refer to \cite[Chapter 5]{BrownEasthamSchmidt13}.
\end{remark}

In the next theorem we strengthen the assumptions by imposing a finite first moment condition
(see \eqref{xx1b} below)
on the differences of the coefficients; note that
\eqref{xx1b} implies \eqref{xx1a} since the coefficients (and their differences) are integrable at $a$.
In this situation it turns out that there appear at most finitely many simple eigenvalues of $H_1$ in each spectral gap and hence
there is no accumulation of eigenvalues to the edges of the band gaps. Concerning the history of this result we refer to the discussion before
the corresponding result on $\mathbb R$, Theorem~\ref{thm2}.

\begin{theorem}\label{thm2a}
Assume that
\begin{align}\label{xx1b}
\int_{a}^\infty \left(\lvert r_1(t)-r_0(t)\rvert + \left\lvert\frac{1}{p_1(t)}-\frac{1}{p_0(t)}\right\rvert + \lvert q_1(t)-q_0(t)\rvert \right)
\vert t\vert\,\mathrm dt <\infty
\end{align}
holds, and let $H_1$ be an arbitrary self-adjoint realization of $\tau_1$ in $L^2((a,\infty);  r_1)$.
Then every gap of $\sigma_{\mathrm{ess}}(H_1)$ contains at most finitely many eigenvalues.
\end{theorem}

In the next result we assume a stronger integrability condition and conclude that the edges of the spectral bands are no embedded eigenvalues of $A_1$;
note that \eqref{xx1c} implies \eqref{xx1b} and \eqref{xx1a}. As pointed out before, this question is important for scattering theory and was
first established by Firsova \cite{F1,F2}  in the case $r_0=r_1=p_0=p_1=1$.

\begin{theorem}\label{thm3a}
Assume that
\begin{align}\label{xx1c}
\int_{a}^\infty \left(\lvert r_1(t)-r_0(t)\rvert + \left\lvert\frac{1}{p_1(t)}-\frac{1}{p_0(t)}\right\rvert + \lvert q_1(t)-q_0(t)\rvert \right)
\vert t\vert^2\,\mathrm dt <\infty
\end{align}
holds, and let $H_1$ be an arbitrary self-adjoint realization of $\tau_1$ in $L^2((a,\infty);  r_1)$.
Then the edges of the spectral bands are no eigenvalues of $H_1$ and the spectral bands consist of purely absolutely continuous spectrum of $H_1$.
\end{theorem}

We find it worthwhile do provide another set of assumptions that also imply absence of eigenvalues at the edges of the spectral bands. Here we only
assume the integrability condition \eqref{xx1b}, but for $r_1$ and $q_1$ additional assumptions are required. It is left to the reader to formulate
a variant of Proposition~\ref{propi} for  the self-adjoint realization $A_1$ of $\tau_1$ in $L^2(\mathbb R;  r_1)$.

\begin{proposition}\label{propi}
Assume that \eqref{xx1b} holds and that there exist positive constants $C_0$, $C_1$ such that $r_1$ and $q_1$ satisfy $C_0 \le r_1(t) \le C_1$ and
$\int_{t-1}^{t+1} |q_1(s)|^2 ds \le C_1$ for $t$ in some neighbourhood of \(\infty\).
Let $H_1$ be an arbitrary self-adjoint realization of $\tau_1$ in $L^2((a,\infty);  r_1)$.
Then the edges of the spectral bands are no eigenvalues of $H_1$ and the spectral bands consist of purely absolutely continuous spectrum of $H_1$.
\end{proposition}

The proofs of Theorem \ref{thm1a}, Theorem~\ref{thm2a}, Theorem~\ref{thm3a}, and Proposition~\ref{propi} are at the end of this section.
In what follows, we investigate solutions of the periodic and the perturbed periodic problem. The first lemma is more or less a variant of standard working knowledge in periodic differential operators and
is essentially contained in \cite[Chapter 1]{BrownEasthamSchmidt13} or \cite{Weidmann87}. For the convenience of the reader we provide a short proof.

\begin{lemma}
	\label{pingpong}
	For $\lambda\in\mathbb R$ there exist linearly independent solutions $u_{0}=u_0(\cdot,\lambda)$ and $v_{0}=v_0(\cdot,\lambda)$ of $(\tau_0-\lambda)u=0$
	and $c=c(\lambda)\in\mathbb C$ such that the functions $U_0=U_0(\cdot,\lambda)$ and $V_0=V_0(\cdot,\lambda)$ given by
		      \begin{equation}
				  \label{quasiperSol1}
				  \begin{split}
				  U_0(x)&=\exp\left(c\frac{x-a}{\omega}\right)\cdot\begin{pmatrix}u_{0}(x)\\ (p_0u_0')(x)\end{pmatrix},\\[0.5\baselineskip]
				  V_0(x)&=\exp\left(-c\frac{x-a}{\omega}\right)\cdot\begin{pmatrix}v_{0}(x)\\ (p_0v_0')(x)\end{pmatrix}
				  \end{split}
			  \end{equation}
	 on $(a,\infty)$ have the following property:
	\begin{itemize}
		\item[(i)] If $\lambda\in \mathbb R\setminus\sigma_{\mathrm{ess}}(H_0)$, then $U_0$ and $V_0$
			  are both $\omega$-periodic and bounded on $(a,\infty)$, where $\operatorname{Re}c>0$.
		\item[(ii)] If $\lambda$ is an interior point of $\sigma_{\mathrm{ess}}(H_0)$, then $U_0$ and $V_0$ are both $\omega$-periodic and bounded on $(a,\infty)$, where $\operatorname{Re}c=0$. In particular, $\lvert u_0\rvert$ and $\lvert v_0\rvert$ are $\omega$-periodic and bounded on $(a,\infty)$.
		\item[(iii)] If $\lambda$ is a boundary point of $\sigma_{\mathrm{ess}}(H_0)$, then $U_0$ is $\omega$-periodic and bounded on $(a,\infty)$, where $\operatorname{Re} c=0$ and, in particular, $\lvert u_0\rvert$ is $\omega$-periodic and bounded on $(a,\infty)$. Furthermore, $V_0$ satisfies
		\begin{equation}
			\label{mint}
			\lVert V_0(x)\rVert_{\mathbb C^2} \leq C \left(1+\frac{x-a}{\omega}\right)
		\end{equation}
		on $(a,\infty)$ for some positive constant $C$.
	\end{itemize}
	In the cases (i) and (iii) the solutions $u_0$ and $v_0$ can be chosen to be real-valued. Moreover, if $\lambda\in\sigma_{\mathrm{ess}}(H_0)$, then for every non-trivial solution of $(\tau_0-\lambda)u=0$ there exists a positive constant $E$ such that
	\begin{equation}
			\label{blin}
			\int_{a+n\omega}^{a+(n+1)\omega} \lvert u(t)\rvert^2 r_0(t)\,\mathrm d t \geq E\quad\text{for all } n\in\mathbb N.
	\end{equation}
\end{lemma}
\begin{proof}
	Let $\mathcal L$ be the two-dimensional complex space of solutions of $(\tau_0-\lambda)u=0$. As the coefficients of $\tau_0$ are $\omega$-periodic,
	for every $f\in\mathcal L$ the function $f(\cdot + \omega)$ is again in $\mathcal L$. Now we identify the linear
	map $\mathcal M: \mathcal L\rightarrow \mathcal L$, $f\mapsto f(\cdot+\omega)$ with the matrix
	\begin{equation*}
		M=\begin{pmatrix}
			\hat u(a+\omega) & \hat v(a+\omega)\\
			(p_0\hat u')(a+\omega) & (p_0\hat v')(a+\omega)
		\end{pmatrix},
	\end{equation*}
	where $\hat u,\hat v\in\mathcal L$ are chosen such that $\hat u(a)=1$, $(p_0\hat u')(a)=0$ and $\hat v(a)=0$, $(p_0\hat v')(a)=1$.
	Since $\det M$ coincides with
the Wronskian the spectrum is
	\begin{equation*}
		\sigma(\mathcal M)=\sigma(M)=\{\mathrm{e}^{c},\mathrm{e}^{-c}\}, \quad\text{where }c\in\mathbb C.
	\end{equation*}
	From now on fix the \emph{Floquet exponent} $c$ such that $\operatorname{Re} c\geq 0$. The eigenvalues $e^{\pm c}$ solve the quadratic equation $\det (M-z) = z^2 - Dz +1=0$, where the Hill discriminant $D:=D(\lambda)= \hat u(a+\omega) + (p_0 \hat v')(a+\omega)$ is real. Therefore,
	\begin{equation}
		\label{pqformel}
		\mathrm{e}^{\pm c} = \frac{D}{2}\pm \sqrt{\frac{D^2}{4}-1}\quad \text{or}\quad \mathrm{e}^{\pm c} = \frac{D}{2}\mp \sqrt{\frac{D^2}{4}-1}.
	\end{equation}
Recall that by \cite[Chapter 12 and Appendix]{Weidmann87} and \cite[Chapter 16]{Weidmann03}
	\begin{equation}
		\label{essspecperiodWeidmann}
		\sigma_{\mathrm{ess}}(H_0) = \{\lambda\in\mathbb R : \lvert D(\lambda)\rvert \leq 2\}\quad\text{and}\quad \partial\sigma_{\mathrm{ess}}(H_0) = \{\lambda\in\mathbb R : \lvert D(\lambda)\rvert = 2\}.
	\end{equation}
  \noindent
  (i) For $\lambda\in \mathbb R\setminus\sigma_{\mathrm{ess}}(H_0)$ we have $\lvert D\rvert >2$ and hence
  $\mathrm{e}^{ c}\neq \mathrm{e}^{- c}$ are both real by \eqref{pqformel}, which leads to $\operatorname{Re} c>0$.
As $\mathcal M$ has two distinct eigenvalues, we find corresponding eigenvectors $u_0, v_0\in\mathcal L$ satisfying
	\begin{align}
		\label{ernie}
		u_0(x+\omega) &= (\mathcal M u_0)(x) = \mathrm{e}^{- c} u_0(x), &(p_0 u_0')(x+\omega) = \mathrm{e}^{- c} (p_0u_0')(x),\\
		\label{bert}
		v_0(x+\omega) &= (\mathcal M v_0)(x) = \mathrm{e}^{c} v_0(x), &(p_0 v_0')(x+\omega) = \mathrm{e}^{c} (p_0v_0')(x),
	\end{align}
	on $(a,\infty)$, where the equalities in \eqref{ernie} and \eqref{bert} for the derivatives follow from the periodicity of $p_0$. From
	\eqref{ernie} and \eqref{bert} one also sees that the functions $U_0$ and $V_0$ defined in \eqref{quasiperSol1} are both $\omega$-periodic, and hence also bounded. This completes the proof of (i).\\
\noindent
  (ii) For an interior point $\lambda$ of $\sigma_{\mathrm{ess}}(H_0)$ we have
	$\lvert D\rvert <2$ by \eqref{essspecperiodWeidmann}, and hence $\mathrm{e}^{ c}$ and $\mathrm{e}^{- c}$ are non-real and complex conjugates of each other,
	which yields $\operatorname{Re} c =0$. As in the proof of (i) $\mathcal M$ has a pair of distinct eigenvalues and we find corresponding eigenvectors $u_0, v_0\in\mathcal L$ satisfying \eqref{ernie}, \eqref{bert}, which shows the periodicity of the $U_0$ and $V_0$ given in \eqref{quasiperSol1} and finishes the proof of (ii).\\
\noindent
    (iii) For $\lambda\in\partial\sigma_{\mathrm{ess}}(H_0)$ we have $\lvert D\rvert =2$ and hence $\mathrm{e}^{ c}=\mathrm{e}^{-c}=D/2\in\{-1,1\}$ by \eqref{pqformel}, and therefore $\operatorname{Re} c =0$. Again, we find $u_0\in\mathcal L$ such that \eqref{ernie} holds and this shows the periodicity of the function $U_0$ defined in \eqref{quasiperSol1}. If the geometric multiplicity of $\mathrm{e}^{ c}=\mathrm{e}^{-c}$ is two, then there is a second linearly independent solution $v_0\in \mathcal L$ which satisfies \eqref{bert}. In this case the function $V_0$ in \eqref{quasiperSol1} is $\omega$-periodic and the estimate \eqref{mint} holds for $C=\sup_{x\in [a,a+\omega]} \lVert V_0(x)\rVert_{\mathbb C^{2}}$. Otherwise, if the geometric multiplicity of $\mathrm{e}^{ c}=\mathrm{e}^{-c}$ is one, then there is a Jordan chain of length two, that is, there exists $v_0\in\mathcal L$ with $\mathcal M v_0 = \mathrm{e}^{c} v_0 + u_0$. One has
	\begin{align}\label{Anitta}
		&v_0(x+\omega) = \mathrm{e}^{c} v_0(x) + u_0(x), & (p_0v_0')(x+\omega) = \mathrm{e}^{c} (p_0v_0')(x) + (p_0u_0')(x)
	\end{align}
	for all $x\in (a,\infty)$. Now consider
\begin{equation*}
		 V_0(x):= \exp\left(-c\frac{x-a}{\omega}\right)\cdot\begin{pmatrix}
 v_0(x)\\  (p_0v_0')(x)
		\end{pmatrix},
	\end{equation*}
as in \eqref{quasiperSol1} and recall that
$\operatorname{Re} c=0$. With \eqref{Anitta} we have
	\begin{equation}
		\label{nalgene}
		\lVert V_0(x+\omega)\rVert_{\mathbb C^2} = \left\lVert\begin{pmatrix}
			\mathrm{e}^{c} v_0(x) + u_0(x)\\ \mathrm{e}^{c} (p_0v_0')(x) + (p_0u_0')(x)
		\end{pmatrix}\right\rVert_{\mathbb C^2} \leq \lVert V_0(x)\rVert_{\mathbb C^2} + \lVert U_0(x)\rVert_{\mathbb C^2}.
	\end{equation}
	Let $x\in (a,\infty)$ and $k\in\mathbb N$ with $k\leq (x-a)/\omega<k+1$. Then \eqref{nalgene}  and the periodicity of $U_0$ give successively
	\begin{equation*}
		\begin{split}
			\lVert V_0(x)\rVert_{\mathbb C^2} &\leq \lVert V_0(x - k\omega)\rVert_{\mathbb C^2} + k \lVert U_0(x-k\omega)\rVert_{\mathbb C^2}\\
			&\leq \lVert V_0(x-k\omega)\rVert_{\mathbb C^2} + \frac{x-a}{\omega} \lVert U_0(x-k\omega)\rVert_{\mathbb C^2}\\
			&\leq \sup_{t\in [a,a+\omega]} \bigl(\lVert V_0(t)\rVert_{\mathbb C^2} + \lVert U_0(t)\rVert_{\mathbb C^2}\bigr) \cdot \left(1+\frac{x-a}{\omega}\right).
		\end{split}
	\end{equation*}
	This shows (iii).
	
	Since in the cases (i) and (iii) the spectrum of $\mathcal M$ is real, $\mathcal M$ can be regarded as a mapping in the real space of real-valued solutions of $(\tau_0-\lambda)u=0$ instead of the complex space $\mathcal L$. Hence, $u_0$ and $v_0$ can be chosen as real-valued solutions. Finally, to show \eqref{blin},
	consider $\lambda\in\sigma_{\mathrm{ess}}(H_0)$ and let $u_0,v_0$ be as in (ii) or (iii). Choose $d_1\in \mathbb C$ such
	that $w_0:= d_1 u_0 + v_0$ is orthogonal to $u_0$  in $L^2((a,a+\omega);r_0)$. We have $\mathcal M u_0 = \mathrm{e}^{-c}u_0$ and $\mathcal Mv_0=\mathrm{e}^{c}v_0 + d_0 u_0$, where $d_0\in\{0,1\}$. Thus,
	\begin{equation*}
		\mathcal M w_0 = \bigl(\mathrm{e}^{-c} d_1 + d_0\bigr) u_0 + \mathrm{e}^{c} v_0 = \bigl(\mathrm{e}^{-c} d_1 + d_0 - \mathrm{e}^{c}d_1\bigr) u_0 + \mathrm{e}^{c} w_0
	\end{equation*}
	and successively for all $n\in\mathbb N$
	\begin{equation*}
		\mathcal M^n w_0 = \gamma_n u_0 + \mathrm{e}^{cn} w_0,\quad\text{where }\gamma_n\in\mathbb C.
	\end{equation*}
	We consider a non-trivial linear combination $\alpha u_0 + \beta w_0$, where $\alpha,\beta\in \mathbb C$. Note that by \eqref{ernie}
	$u_0(t+n\omega) =(\mathcal M^n u_0)(t)= \mathrm{e}^{-nc} u_0(t)$ for $t\in [a,\infty)$ and $n\in\mathbb N$. Recall also
	that $\operatorname{Re} c =0$. If $\beta =0$, then
	\begin{equation*}
		\int_{a+n\omega}^{a+(n+1)\omega} \lvert \alpha u_0(t)\rvert^2 r_0(t)\,\mathrm dt = \int_{a}^{a+\omega} \lvert \alpha u_0(t)\rvert^2 r_0(t)\,\mathrm dt > 0
	\end{equation*}
	for all $n\in\mathbb N$.
	Otherwise, if $\beta\neq 0$, then
	\begin{equation*}
		\begin{split}
			&\int_{a+n\omega}^{a+(n+1)\omega} \lvert \alpha u_0(t) +\beta w_0(t)\rvert^2 r_0(t)\,\mathrm dt\\
			&\qquad = \int_{a}^{a+\omega} \lvert \alpha (\mathcal M^n u_0)(t) +\beta (\mathcal M^n w_0)(t)\rvert^2 r_0(t)\,\mathrm dt\\
			& \qquad = \int_{a}^{a+\omega} \lvert (\alpha \mathrm{e}^{-cn} + \beta\gamma_n )  u_0(t) +\beta \mathrm{e}^{cn}w_0(t)\rvert^2 r_0(t)\,\mathrm dt\\
			&\qquad \geq \int_{a}^{a+\omega} \lvert \beta w_0(t)\rvert^2 r_0(t)\,\mathrm dt>0
		\end{split}
	\end{equation*}
	for all $n\in\mathbb N$. In both cases we conclude \eqref{blin} and Lemma~\ref{pingpong} is shown.
\end{proof}

The solution's asymptotics are basically preserved under $L^1$-perturbations of $\tau_0$ with respect to its coefficients. This is the content of the next lemma.

\begin{lemma}
        \label{lachs}
        Let $\lambda\in\mathbb R$, assume that \eqref{xx1a} holds and let
        $u_0$, $v_0$ and $c$ be as in Lemma~\ref{pingpong}. Then there exist
        linearly independent solutions $u_{1}=u_1(\cdot,\lambda)$ and
$v_{1}=v_1(\cdot,\lambda)$ of $(\tau_1-\lambda)u=0$ such that the
following holds:
        \begin{enumerate}
                \item If $\lambda\in \mathbb R\setminus\sigma_{\mathrm{ess}}(H_0)$,
that is, $\operatorname{Re} c>0$, then
                          \begin{equation}
                                \label{pixel}
                                \exp\left(\operatorname{Re} c
\frac{x-a}{\omega}\right)\cdot\left\lVert\begin{pmatrix}u_1(x)\\
(p_1u_1')(x)\end{pmatrix} - \begin{pmatrix}u_0(x)\\
(p_0u_0')(x)\end{pmatrix}\right\rVert_{\mathbb C^{2}} \rightarrow
0\quad\text{as }x\rightarrow \infty
                      \end{equation}
                          and
                          \begin{equation}
                          \begin{split}
                                \label{phone}
                                &\left\lVert\begin{pmatrix}u_1(x)\\
(p_1u_1')(x)\end{pmatrix}\right\rVert_{\mathbb C^2}\leq
C\exp\left(-\operatorname{Re} c \frac{x-a}{\omega}\right),\\[1ex]
&\left\lVert\begin{pmatrix}v_1(x)\\
(p_1v_1')(x)\end{pmatrix}\right\rVert_{\mathbb C^2}\leq
C\exp\left(\operatorname{Re} c \frac{x-a}{\omega}\right)
                                \end{split}
                      \end{equation}
                      on $(a,\infty)$, where $C=C(\lambda)$ is a positive constant. In particular, $u_1$ is bounded
on $(a,\infty)$.

                \item If $\lambda$ is an interior point of
$\sigma_{\mathrm{ess}}(H_0)$, that is, $\operatorname{Re} c=0$, then
\eqref{pixel} and \eqref{phone} hold
                on $(a,\infty)$, and
                          \begin{equation}\label{bbb}
                                \left\lVert\begin{pmatrix}v_1(x)\\ (p_1v_1')(x)\end{pmatrix} -
\begin{pmatrix}v_0(x)\\ (p_0v_0')(x)\end{pmatrix}\right\rVert_{\mathbb
C^{2}} \rightarrow 0\quad\text{as }x\rightarrow \infty.
                          \end{equation}
                          In particular, $u_1$ and $v_1$ are bounded on $(a,\infty)$.
                \item If $\lambda$ is a boundary point of $\sigma_{\mathrm{ess}}(H_0)$,
that is, $\operatorname{Re} c =0$, and \eqref{xx1b} (and hence also \eqref{xx1a})  holds, then $u_1$ satisfies \eqref{pixel}
                        and the first inequality in \eqref{phone} on $(a,\infty)$. In particular, $u_1$ is bounded
on $(a,\infty)$. If \eqref{xx1c} (and hence also \eqref{xx1a} and \eqref{xx1b}) holds,
			then $v_1$ satisfies \eqref{bbb}.
        \end{enumerate}
        The solutions in (i) and (iii) can be chosen to be real-valued.
\end{lemma}

\begin{proof}
	Let $\lambda\in\mathbb R$. We consider the systems $\phi'=A\phi$ and $\xi'=(A+B)\xi$ corresponding to $(\tau_0-\lambda)u=0$ and $(\tau_1-\lambda)u=0$, respectively, where
	\begin{align*}
		A=\begin{pmatrix}
			0 & \frac{1}{p_0} \\ q_0 - \lambda r_0 & 0
		\end{pmatrix}\quad\text{and}\quad
		B=\begin{pmatrix}
		0 & \frac{1}{p_1}-\frac{1}{p_0} \\ (q_1-q_0) - \lambda (r_1-r_0) & 0
		\end{pmatrix}.
	\end{align*}
	From \eqref{xx1a} we obtain $\lVert B(\cdot)\rVert_{\mathbb C^{2\times 2}}\in L^1(a,\infty)$. With $u_0$ and $v_0$ from Lemma~\ref{pingpong}
	we consider the fundamental solution $\Phi$ of the system $\phi'=A\phi$ given by
	\begin{equation}
	\label{pillepalle}
	\Phi(x) = \begin{pmatrix}
		u_0(x) & v_0(x)\\ (p_0u_0')(x) & (p_0 v_0')(x)
		\end{pmatrix},\quad x\in (a,\infty),
	\end{equation}
	so that
	\begin{equation*}
		\bigl(\Phi(t)\bigr)^{-1} = \frac{1}{W(u_0,v_0)}\begin{pmatrix}
			(p_0v_0')(t) & -v_0(t)\\ -(p_0u_0')(t) & u_0(t)
			\end{pmatrix},\quad t\in (a,\infty),
	\end{equation*}
where $W$ is the Wronskian.
	With \eqref{quasiperSol1} in Lemma~\ref{pingpong} we estimate for all $x,t\in[a,\infty)$
	\begin{equation}
	\label{BigPhi}
		\lVert\Phi(x)(\Phi(t))^{-1}\rVert_{\mathbb C^{2\times 2}} \leq \tilde E \mathrm e^{\operatorname{Re} c \frac{t-x}{\omega}} \lVert U_0(x)\rVert_{\mathbb C^2} \lVert V_0(t)\rVert_{\mathbb C^2} + \tilde E \mathrm e^{\operatorname{Re} c \frac{x-t}{\omega}} \lVert U_0(t)\rVert_{\mathbb C^2} \lVert V_0(x)\rVert_{\mathbb C^2}
	\end{equation}
	where $\tilde E$ is a suitable positive constant.
	
	We show (i) and (ii). In this case, $\operatorname{Re} c \geq 0$ and $U_0$, $V_0$ are bounded. We consider the Banach space $\mathcal B$
	of all continuous $\mathbb C^2$-valued functions with exponential decay of order $-\operatorname{Re} c/\omega$, that is,
$$
\mathcal B :=
\left\{\xi : [a,\infty) \to \mathbb C^2\,\mbox{continuous}:\lVert\xi(x) \rVert_{\mathbb C^{2}} \leq \gamma
\mathrm{e}^{-\operatorname{Re} c\frac{x}{\omega}} \mbox{ for some }\gamma\geq 0 \mbox{ on } [a,\infty)
\right\}
$$
and the corresponding norm
	\begin{align*}
		\lVert\xi\rVert_{\mathcal B}:=\sup_{x\in[a,\infty)} \mathrm{e}^{\operatorname{Re} c\frac{x-a}{\omega}}\lVert\xi(x) \rVert_{\mathbb C^{2}}<\infty.
	\end{align*}
	For $\xi\in\mathcal B$ we define	
	\begin{equation}
		\label{volterra}
		(T \xi)(x) := -\Phi(x) \int_{x}^\infty (\Phi(t))^{-1} B(t) \xi(t)\,\mathrm dt,\quad x\in [a,\infty).
	\end{equation}
The integral in \eqref{volterra} converges. Indeed, the estimate in \eqref{BigPhi} yields
	\begin{equation}
		\label{BigPhi2}
		\lVert\Phi(x)(\Phi(t))^{-1}\rVert_{\mathbb C^{2\times 2}} \leq E\mathrm{e}^{\operatorname{Re}c \frac{t-x}{\omega}}
	\end{equation}
	for $a\leq x\leq t<\infty$, where $E$ is a suitable positive constant.
Then \eqref{volterra} with \eqref{BigPhi2} give
\begin{equation}\label{jaok}
		\begin{split}
\mathrm{e}^{\operatorname{Re}c \frac{x-a}{\omega}}
\lVert(T\xi)(x)\rVert_{\mathbb C^{2}}
& \leq
\mathrm{e}^{\operatorname{Re}c \frac{x-a}{\omega}} \int_x^\infty
E\mathrm{e}^{\operatorname{Re}c \frac{t-x}{\omega}}
 \lVert B(t)\rVert_{\mathbb C^{2\times 2}}
  \lVert \xi(t)\rVert_{\mathbb C^{2}} \mathrm dt\\[1ex]
  & \leq
  \lVert\xi\rVert_{\mathcal B}\, E  \int_x^\infty
\lVert B(t)\rVert_{\mathbb C^{2\times 2}} \mathrm dt<\infty
		\end{split}
\end{equation}
and hence the integral in \eqref{volterra} exists. Moreover, we also conclude that $T\xi\in\mathcal B$ and $T$ is a bounded everywhere defined operator in $\mathcal B$.

We claim that for $n\in\mathbb N$ the estimate
	\begin{equation}
		\label{dot}
		\lVert(T^n\xi)(x)\rVert_{\mathbb C^{2}} \leq
		\mathrm{e}^{-\operatorname{Re}c \frac{x-a}{\omega}}
		\lVert\xi\rVert_{\mathcal B} \frac{1}{n!}\left(E\int_{x}^\infty \lVert B(t)\rVert_{\mathbb C^{2\times 2}}\,\mathrm dt\right)^n\!\!,\quad x\in[a,\infty),
	\end{equation}
	holds. In fact, for $n=1$ this is true by \eqref{jaok}.
	Now assume that \eqref{dot} holds for some $n\in \mathbb N$. We set
	$G(t):= \frac{1}{n+1}\left(E\int_{t}^\infty \lVert B(s)\rVert_{\mathbb C^{2\times 2}}\,\mathrm ds\right)^{n+1}$ and compute
\begin{equation*}
		\begin{split}
\lVert(T^{n+1}\xi)(x)\rVert_{\mathbb C^{2}} &\leq
 \int_x^\infty
 \lVert\Phi(x)(\Phi(t))^{-1}\rVert_{\mathbb C^{2\times 2}}
 \lVert B(t)\rVert_{\mathbb C^{2\times 2}}
  \lVert (T^n\xi)(t)\rVert_{\mathbb C^{2}} \mathrm dt\\[1ex]
  & \leq  \mathrm{e}^{-\operatorname{Re}c \frac{x-a}{\omega}}
  \int_x^\infty E
 \lVert B(t)\rVert_{\mathbb C^{2\times 2}}
  \lVert \xi\rVert_{\mathcal B}
 \frac{1}{n!}\left(E\int_{t}^\infty \lVert B(s)\rVert_{\mathbb C^{2\times 2}}\,\mathrm ds\right)^n
  \mathrm dt\\[1ex]
   & =  \mathrm{e}^{-\operatorname{Re}c \frac{x-a}{\omega}}
     \lVert \xi\rVert_{\mathcal B}  \frac{1}{n!}
  \int_x^\infty -G^\prime(t) \mathrm dt\\[1ex]
		 &=\mathrm{e}^{-\operatorname{Re}c \frac{x-a}{\omega}}
		\lVert\xi\rVert_{\mathcal B} \frac{1}{(n+1)!}\left(E\int_{x}^\infty \lVert B(t)\rVert_{\mathbb C^{2\times 2}}\,\mathrm dt\right)^{n+1}
		\end{split}
\end{equation*}
which shows \eqref{dot} for any $n\in\mathbb N$. Hence,
	\begin{equation*}
		\lVert T^n \xi \rVert_{\mathcal B} \leq \lVert\xi\rVert_{\mathcal B} \frac{1}{n!}\left(E\int_{a}^\infty \lVert B(t)\rVert_{\mathbb C^{2\times 2}}\,\mathrm dt\right)^n
	\end{equation*}
	and the Neumann series $(I-T)^{-1} = \sum_{n\in\mathbb N} T^n$ converges in the operator norm induced by $\lVert \cdot\rVert_{\mathcal B}$. Observe that
	for a solution $\phi\in\mathcal B$ of $\phi'=A\phi$ the function $\xi := (I-T)^{-1}\phi\in\mathcal B$ satisfies $\xi' = (A+B)\xi$ since
	\begin{equation}
	\label{klu}
		\xi = T\xi + \phi
	\end{equation}
	yields
	\begin{equation}
		\label{tom}
		\xi' = \Phi' \Phi^{-1} T\xi + B\xi + \phi' = A(T\xi + \phi) + B\xi = (A+B)\xi.
	\end{equation}
	Furthermore,  from \eqref{klu} and \eqref{jaok} we also conclude
	\begin{equation}
		\label{jerry}
		\mathrm{e}^{\operatorname{Re}c \frac{x-a}{\omega}}\lVert \phi(x) - \xi(x)\rVert_{\mathbb C^2} \rightarrow 0\quad\text{as }x\rightarrow \infty.
	\end{equation}
	
	Now let us consider the continuous function $(u_0,p_0 u_0')^\top:[a,\infty)\rightarrow\mathbb C^2$. According to Lemma \ref{pingpong} (i)--(ii) we have
  $(u_0,p_0 u_0')^\top\in\mathcal B$. From the above considerations we see that $(I-T)^{-1} (u_0,p_0 u_0')^\top$ is a solution of $\xi' = (A+B)\xi$ and hence
  \begin{equation*}
  \begin{pmatrix}u_1\\
p_1u_1'\end{pmatrix} :=
(I-T)^{-1}
  \begin{pmatrix}u_0\\
p_0u_0'\end{pmatrix}\in\mathcal B
  \end{equation*}
  gives a solution $u_1$ of $(\tau_1-\lambda)u=0$ such that the assertions in (i) and (ii) hold
  for $u_1$; note that \eqref{jerry} implies \eqref{pixel} and $(u_1,p_1 u_1')^\top\in\mathcal B$ shows the first inequality in \eqref{phone}.
  Observe, that if $\lambda$ is an interior point of $\sigma_{\mathrm{ess}}(H_0)$ then also $(v_0,p_0v_0')^\top \in\mathcal B$ by Lemma~\ref{pingpong}~(ii) as $\operatorname{Re} c =0$. Again it follows that
    $$
  \begin{pmatrix}v_1\\
p_1v_1'\end{pmatrix} := (I-T)^{-1}
  \begin{pmatrix}v_0\\
p_0v_0'\end{pmatrix}\in\mathcal B
  $$
  gives a solution $v_1$ of $(\tau_1-\lambda)u=0$ and \eqref{bbb} follows from \eqref{jerry}. Thus we have shown (ii) and it remains to
  check in (i) the second inequality in \eqref{phone}. In fact, for any solution $v_1$ of $(\tau_1-\lambda)u=0$ and $\xi=(v_1,p_1v_1')^\top$ one has
	\begin{equation*}
		\xi(x) = \Phi(x) \left(\bigl(\Phi(a)\bigr)^{-1}\xi(a) + \int_a^x \bigl(\Phi(t)\bigr)^{-1} B(t) \xi(t)\,\mathrm dt\right).
	\end{equation*}
	From \eqref{BigPhi} we obtain $\lVert\Phi(x)(\Phi(t))^{-1}\rVert_{\mathbb C^{2\times 2}} \leq E\mathrm{e}^{\operatorname{Re}c \frac{x-t}{\omega}}$
	for $a\leq t\leq x<\infty$ (cf. \eqref{BigPhi2}) with some $E>0$. Hence,
	\begin{equation*}
		\mathrm{e}^{-\operatorname{Re}c \frac{x-a}{\omega}}\lVert \xi(x)\rVert_{\mathbb C^2} \leq E\lVert\xi(a)\rVert_{\mathbb C^2} + E \int_a^x\lVert B(t)\rVert_{\mathbb C^{2\times 2}} \left( \mathrm{e}^{-\operatorname{Re}c \frac{t-a
		}{\omega}} \lVert \xi(t)\rVert_{\mathbb C^2}\right)\,\mathrm dt
	\end{equation*}
	for all $x\in [a,\infty)$. Now Gronwall's inequality yields
	\begin{equation*}
		\lVert \xi(x)\rVert_{\mathbb C^2} \leq \mathrm{e}^{\operatorname{Re}c \frac{x-a
		}{\omega}} E \lVert \xi(a)\rVert_{\mathbb C^{2}}
		\mathrm e^{E \int_a^x \lVert B(t)\rVert_{\mathbb C^{2\times 2}}\,\mathrm dt},
	\end{equation*}
	and hence the second inequality in \eqref{phone} holds for any solution $v_1$ of $(\tau_1-\lambda)u=0$. This completes the proof of (i) and (ii).

	We prove (iii). In the case $\lambda\in\mathbb \partial \sigma_{\mathrm{ess}}(H_0)$ Lemma~\ref{pingpong}~(iii) implies $\operatorname{Re} c =0$
	and the Banach space  $\mathcal B$ from above is the usual space of bounded continuous functions.
	Let $\phi\in\mathcal B$ and let $T$ be as in \eqref{volterra}.
	From Lemma~\ref{pingpong}~(iii) and \eqref{BigPhi} we obtain
	\begin{equation}\label{BigPhi3}
		\lVert\Phi(x)(\Phi(t))^{-1}\rVert_{\mathbb C^{2\times 2}}\leq E \left(1+\frac{t-a}{\omega}\right)
	\end{equation}
	for $a\leq x\leq t<\infty$ and hence
	\begin{equation}
		\label{pointqq}
		\lVert(T\phi)(x)\rVert_{\mathbb C^2}\leq \lVert\phi\rVert_{\mathcal B}\, E\int_x^\infty \left(1+ \frac{t-a}{\omega}\right)\lVert B(t)\rVert_{\mathbb C^{2\times 2}}\,\mathrm dt,
	\end{equation}
	where the integral converges since $(1+\vert \cdot \vert)\lVert B(\cdot)\rVert_{\mathbb C^{2\times 2}}\in L^1(a,\infty)$ by  \eqref{xx1b}.
	In the same way as in the proof of (i) and (ii) one verifies with $G(t)$ replaced by
	$H(t)=\frac{1}{n+1}\left(E\int_{t}^\infty (1+\frac{s-a}{\omega})\lVert B(s)\rVert_{\mathbb C^{2\times 2}}\,\mathrm ds\right)^{n+1}$ that
	\begin{equation*}
		\lVert(T^n\phi)(x)\rVert_{\mathbb C^2}\leq \lVert\phi\rVert_{\mathcal B} \frac{1}{n!}\left(E\int_x^\infty \left(1+ \frac{t-a}{\omega}\right)\lVert B(t)\rVert_{\mathbb C^{2\times 2}}\,\mathrm dt\right)^n
	\end{equation*}
	and
	\begin{equation*}
		\lVert(T^n\phi)\rVert_{\mathcal B}\leq \lVert\phi\rVert_{\mathcal B} \frac{1}{n!}\left(E\int_a^\infty \left(1+ \frac{t-a}{\omega}\right)\lVert B(t)\rVert_{\mathbb C^{2\times 2}}\,\mathrm dt\right)^n	
	\end{equation*}
	hold for all $n\in\mathbb N$ and $x\in[a,\infty)$. As above it follows that $(I-T)^{-1}$ is an everywhere defined bounded operator in $\mathcal B$
	and for a solution $\phi\in\mathcal B$ of $\phi'=A\phi$ the function $\xi=(I-T)^{-1}\phi\in\mathcal B$ satisfies \eqref{klu} and \eqref{tom}.
	Hence it follows from \eqref{pointqq} that \eqref{jerry} holds with $\operatorname{Re} c =0$.
	Now consider $(u_0,p_0u_0')^\top$, which is in $\mathcal B$ by Lemma~\ref{pingpong}~(iii), and set
	 \begin{equation}\label{hirter}
  \begin{pmatrix}u_1\\
p_1u_1'\end{pmatrix} :=
(I-T)^{-1}
  \begin{pmatrix}u_0\\
p_0u_0'\end{pmatrix}\in\mathcal B.
  \end{equation}
	Then $u_1$ is a solution of $(\tau_1-\lambda)u=0$ and the assertions for $u_1$ in (iii) follow.
	
	Now assume that the integrability condition \eqref{xx1c} (and hence also \eqref{xx1a} and \eqref{xx1b}) holds. Then $(1+\vert\cdot\vert^2)\lVert B(\cdot)\rVert_{\mathbb C^{2\times 2}}\in L^1(a,\infty)$ and for
    continuous functions $\xi:[a,\infty)\rightarrow\mathbb C^2$ such that
	\begin{equation}\label{popl}
		C_\xi:=\sup_{x\in[a,\infty)} \left(1+\tfrac{x-a}{\omega}\right)^{-1} \lVert\xi(x) \rVert_{\mathbb C^{2}}<\infty
	\end{equation}
	we can consider the integral \eqref{volterra}, where we shall use the notation $\widetilde T$ to distinguish from the operator $T$ acting in the
	Banach space $\mathcal B$.
	In fact, by \eqref{BigPhi3} we have
	\begin{equation}
		\label{kvatch}
		\lVert(\widetilde T\xi)(x)\rVert_{\mathbb C^2}\leq E \, C_\xi\, \int_x^\infty \left(1+ \frac{t-a}{\omega}\right)^2\lVert B(t)\rVert_{\mathbb C^{2\times 2}}\,\mathrm dt
	\end{equation}
	for $x\in(a,\infty)$ and hence $\widetilde T\xi\in\mathcal B$.
	Now let $\phi = (v_0,p_0v_0')^\top$ and observe that by Lemma~\ref{pingpong}~(iii) $\phi$ satisfies an estimate of the form \eqref{popl}.
	The function $\xi:= (I-T)^{-1} \widetilde T \phi + \phi$ also satisfies \eqref{popl} and $\xi-\phi = (I-T)^{-1}\widetilde T\phi\in\mathcal B$ . Hence,
	\begin{equation*}
		\widetilde T\phi =(I-T)(\xi-\phi) = (\xi-\phi) - \widetilde T(\xi-\phi) = \xi -\phi - \widetilde T\xi + \widetilde T\phi,
	\end{equation*}
	which implies
	\begin{equation}\label{ay}
		\xi = \phi + \widetilde T\xi.
	\end{equation}
	 As in \eqref{tom} we see that $\xi$ solves $\xi' = (A+B)\xi$ and hence $\xi=(v_1,p_1v_1')^\top$ with some solution $v_1$ of $(\tau_1-\lambda)u=0$. From \eqref{kvatch} and \eqref{ay} we obtain $\lVert \phi(x) - \xi(x)\rVert_{\mathbb C^2} \rightarrow 0$ as $x\rightarrow \infty$, which shows \eqref{bbb}.
	 To see that $v_1$ and $u_1$ in the present situation are linearly independent assume the contrary. Then also
	 $(v_1,p_1v_1')^\top$ and $(u_1,p_1u_1')^\top$ are multiples of each other and hence $(v_1,p_1v_1')^\top\in\mathcal B$. But then also
	 $$(I-\widetilde T)(v_1,p_1v_1')^\top = (I-T)(v_1,p_1v_1')^\top = (v_0,p_0v_0')^\top$$ and
	 $(I-T)(u_1,p_1u_1')^\top =(u_0,p_0u_0')^\top$ (see \eqref{hirter}) are multiples of each other; a contradiction.

	Note that in the cases (i) and (iii) the solutions $u_0$ and $v_0$ from Lemma~\ref{pingpong} can be chosen to be real-valued.
	Then $\Phi$ in \eqref{pillepalle} has values in $\mathbb R^{2\times 2}$ and the solution $u_1$ and $v_1$ in (i) and (iii)
	constructed via $T$ in \eqref{volterra} are also real-valued.
\end{proof}

\begin{proof}[Proof of Theorem~\ref{thm1a}]
	For $\lambda\in\mathbb R$ let $c=c(\lambda)$ and $u_j=u_j(\cdot,\lambda), v_j(\cdot,\lambda)$, $j=0,1$, be as in Lemma~\ref{pingpong} and Lemma~\ref{lachs}. The proof is divided into four steps.
	\vskip 0.2cm\noindent
	\emph{Step 1.} Let $\lambda$ be an arbitrary element of the non-empty
interior of $\sigma_{\mathrm{ess}}(H_0)$, that is, $\operatorname{Re} c
= 0$ by Lemma~\ref{pingpong}~(ii). We show that for every nonzero
solution $w_1$ of $(\tau_1-\lambda)u=0$ there exist positive constants
$E_1$ and $E_2$ such that
        \begin{align}
        \label{jenawest}
        E_1\leq\int_{a+n\omega}^{a+(n+1)\omega} \vert w_1(t)\rvert^2
r_1(t)\,\mathrm dt\leq E_2
        \end{align}
        holds for all sufficiently large $n\in\mathbb N$. Fix an arbitrary
nontrivial linear combination $w_1=\alpha u_1 + \beta v_1$,
$\alpha,\beta\in\mathbb C$. For the same constants $\alpha$ and $\beta$
let $w_0=\alpha u_0+\beta v_0$. From Lemma~\ref{pingpong}~(ii) and the periodicity of $U_0$ and $V_0$
  we obtain for $n\in\mathbb N$ and $t\in [a,\infty)$
  $$
  u_0(t+n\omega) =\mathrm{e}^{-nc} u_0(t)\quad\mbox{and}\quad
   v_0(t+n\omega) =\mathrm{e}^{nc} v_0(t).
  $$
 This, $\lvert \mathrm{e}^{-nc}\rvert = \lvert\mathrm{e}^{nc}\rvert =1$, and the periodicity of $r_0$ imply
        \begin{align*}
                \int_{a+n\omega}^{a+(n+1)\omega} \lvert w_0(t)\rvert^2r_0(t)\,\mathrm
dt &\leq  2\int_{a+n\omega}^{a+(n+1)\omega} \bigl(\lvert\alpha u_0(t)\rvert^2+ \lvert\beta
v_0(t)\rvert^2\bigr)r_0(t)\,\mathrm dt\\
&= 2 \int_{a}^{a+\omega} \bigl(\lvert\alpha
\mathrm{e}^{-nc} u_0(t)\rvert^2+\lvert\beta \mathrm{e}^{nc}
v_0(t)\rvert^2\bigr)r_0(t)\,\mathrm dt,
        \end{align*}
and hence together with \eqref{blin} we conclude
        \begin{align*}
                E\leq \int_{a+n\omega}^{a+(n+1)\omega} \lvert
w_0(t)\rvert^2r_0(t)\,\mathrm dt\leq E'
        \end{align*}
        for some $E,E'>0$ and all $n\in\mathbb N$. Furthermore,
        \begin{equation}
        \label{takkitakki}
        \begin{split}
        \bigl\lvert \lvert w_1\rvert^2r_1 - \lvert w_0\rvert^2r_0\bigr\rvert &=
\bigl\lvert \lvert w_1\rvert^2 (r_1-r_0) + \bigl(\lvert w_1\rvert^2 -
\lvert w_0\rvert^2\bigr)r_0\bigr\rvert\\
        &\leq \lvert w_1\rvert^2\lvert r_1-r_0\rvert + \lvert w_1
-w_0\rvert\bigl(\lvert w_1\rvert + \lvert w_0\rvert\bigr) r_0
        \end{split}
        \end{equation}
        holds pointwise a.\,e.\ on $(a,\infty)$.
        By Lemma~\ref{pingpong}~(ii) and Lemma~\ref{lachs}~(ii) the solutions
$w_0$, $w_1$ are bounded and
        \begin{equation*}
                \lvert w_1(x)-w_0(x)\rvert \leq \lvert \alpha\rvert \cdot\lvert
u_1(x)-u_0(x)\rvert + \lvert\beta\rvert \cdot \lvert v_1(x)-v_0(x)\rvert
\rightarrow 0,\quad\text{as } x\rightarrow\infty
        \end{equation*} by \eqref{pixel} and \eqref{bbb}. Thus,
\eqref{takkitakki} together with $r_1-r_0\in L^1(a,\infty)$ and the
periodicity of $r_0$ imply the existence of $n_0\in\mathbb N$ such that
        \begin{align*}
                \left\lvert\int_{a+n\omega}^{a+(n+1)\omega} \lvert w_1(t)\rvert^2
r_1(t)\,\mathrm dt - \int_{a+n\omega}^{a+(n+1)\omega} \lvert
w_0(t)\rvert^2 r_0(t)\,\mathrm dt\right\rvert\leq \frac{E}{2}
        \end{align*}
        for all $n\geq n_0$. Choosing $E_1=\frac{E}{2}$ and $E_2=E' +
\frac{E}{2}$ shows \eqref{jenawest} for all $n\geq n_0$.

        As an immediate consequence, $\tau_1$ is in the limit-point case at
$\infty$ and no non-trivial solution of $(\tau_1-\lambda)u=0$ is in
$L^2((a,\infty);r_1)$,
        and thus $\lambda\in\sigma_{\mathrm{ess}}(H_1)$; cf.\
\cite[Theorem~11.5]{Weidmann87}. Since the essential spectra are closed
sets we obtain
        \begin{equation*}
                \sigma_{\mathrm{ess}}(H_0)\subset \sigma_{\mathrm{ess}}(H_1).
        \end{equation*}
\noindent
\emph{Step 2.} Let $\lambda$ be an arbitrary element of the non-empty
interior of $\sigma_{\mathrm{ess}}(H_0)$.
        We prove now the statement on the absolute continuous spectrum of $H_1$.
	   A non-trivial solution $u$ of $(\tau_1-\lambda)u=0$ for real $\lambda$ is called \emph{sequentially subordinant} at $\infty$ with respect to another non-trivial solution $v$ of $(\tau_1-\lambda)u=0$ if
	\begin{equation*}
		\liminf_{x\rightarrow \infty} \frac{\int_a^x \lvert u(t)\rvert^2 r_1(t)\,\mathrm d t}{\int_a^x  \lvert v(t)\rvert^2 r_1(t)\,\mathrm d t}=0,
	\end{equation*}
see \cite[Section~9.5]{te} and also \cite{sst}.
	By \eqref{jenawest} in the first step of proof above we see that for all interior points $\lambda$ of $\sigma_{\mathrm{ess}}(H_1)$ no sequentially subordinate solution of $(\tau_1-\lambda)u=0$ exists. Standard subordinancy theory (cf.\ Theorem~9.27 together with the remark below in \cite{te}) implies that the absolutely continuous spectrum of $H_1$ equals $\sigma_{\mathrm{ess}}(H_1)$ and the interior of $\sigma_{\mathrm{ess}}(H_1)$ is purely absolutely continuous.
\vskip 0.2cm\noindent
	\emph{Step 3.} We proceed to prove the converse inclusion $\sigma_{\mathrm{ess}}(H_1)\subset \sigma_{\mathrm{ess}}(H_0)$.
	Suppose $\lambda \not\in \sigma_{\mathrm{ess}}(H_0)$, that is, $\operatorname{Re} c> 0$
by Lemma~\ref{pingpong}~(i).
By Lemma~\ref{lachs}~(i)
	there exist real-valued solutions $u_1$ and $v_1$. For $g\in L^2((a,\infty);r_1)$ set
	\begin{equation*}
		(Sg)(x):=\frac{1}{W(u_1,v_1)} \int_{a}^\infty G(x,t) g(t) r_1(t)\,\mathrm d t, \quad 	G(x,t) := \begin{cases}
			u_1(x)v_1(t) & \text{if }a\leq t\leq x, \\
			u_1(t)v_1(x) & \text{if }a\leq x\leq t,
		\end{cases}
	\end{equation*}
	that is
	\begin{equation}
		\label{S}
		(Sg)(x) = \frac{1}{W(u_1,v_1)} \left(u_1(x)\int_a^x v_1(t)g(t)r_1(t)\,\mathrm d t + v_1(x)\int_x^\infty u_1(t)g(t)r_1(t)\,\mathrm d t \right),
	\end{equation}
where $W$ stands again for the Wronskian.
	Define
	\begin{equation*}
		E:=\sup_{n\in \mathbb N}\int_{a+n\omega}^{a+(n+1)\omega} r_1(t)\,\mathrm d t,
	\end{equation*}
	which is finite since $r_0-r_1\in L^1(a,\infty)$ and $r_0$ is periodic and locally integrable.
	Consider an arbitrary $x\in [a,\infty)$. By \eqref{phone} in Lemma~\ref{lachs}~(i)
	\begin{equation*}
		\int_{a}^\infty \lvert G(x,t) \rvert r_1(t)\,\mathrm d t \leq C^2\left(\int_{a}^x \mathrm{e}^{\operatorname{Re} c\frac{t-x}{\omega}} r_1(t)\,\mathrm d t + \int_{x}^\infty \mathrm{e}^{\operatorname{Re}c\frac{x-t}{\omega}} r_1(t)\,\mathrm d t\right).
	\end{equation*}
	Let $k\in\mathbb N$ with $k\omega +a \leq x < (k+1)\omega +a$. We continue estimating
	\begin{equation*}
		\begin{aligned}
			\int_{a}^\infty \lvert G(x,t) \rvert r_1(t)\,\mathrm d t \leq{ } & C^2\sum_{n=0}^{k} \mathrm{e}^{\operatorname{Re}c\cdot(1-n)} \int_{a+(k-n)\omega}^{a+(k+1-n)\omega} r_1(t)\,\mathrm d t\\[0.5\baselineskip]
			&{ }+ C^2\sum_{n=0}^{\infty} \mathrm{e}^{\operatorname{Re}c\cdot(1-n)} \int_{a+(n+k)\omega}^{a+(n+1+k)\omega} r_1(t)\,\mathrm d t \\[0.5\baselineskip]
			\leq{ } & 2C^2 E\sum_{n=0}^\infty \mathrm{e}^{\operatorname{Re}c\cdot(-n+1)}<\infty.
		\end{aligned}
	\end{equation*}
	Due to the symmetry $G(x,t)=G(t,x)$ the same bound holds for $\int_{a}^\infty\lvert G(x,t) \rvert r_1(x)\,\mathrm d x$ evaluated at $t\in [a,\infty)$. As a consequence of the Schur criterion (see, e.\,g., \cite[Lemma~0.32]{te}) one obtains that $S$ is a bounded operator in $L^2((a,\infty); r_1)$. For $g\in L^2((a,\infty); r_1)$ a straightforward calculation using \eqref{S} and $(\tau_1-\lambda)u_1=(\tau_1-\lambda)v_1=0$ shows that $Sg$, $p_1 (Sg)'$ are absolutely continuous
	on $(a,\infty)$, and that $Sg$ solves the inhomogeneous differential equation $(\tau_1-\lambda)u=g$. Thus, $\tau_1 (Sg)= \lambda Sg + g\in L^2((a,\infty); r_1)$ and hence
	$Sg$ is in the domain of the maximal operator
associated to  $\tau_1$ in $L^2((a,\infty); r_1)$ and $S$ is injective. Moreover,
since $u_1$ and $v_1$ are real-valued it follows that $S$ is self-adjoint, so that $S^{-1}$ is a self-adjoint restriction
of the maximal operator associated with $\tau_1-\lambda$. In other words, $S$ is the resolvent at $\lambda$ of some self-adjoint realization of $\tau_1$ and
as all self-adjoint
realizations of $\tau_1$ have the same essential spectrum, we obtain
$$
\lambda \notin \sigma_{\mathrm{ess}}(H_1).
$$
Thus $\sigma_{\mathrm{ess}}(H_1)\subset \sigma_{\mathrm{ess}}(H_0)$ and together with the first step
	\begin{equation*}\sigma_{\mathrm{ess}}(H_1)= \sigma_{\mathrm{ess}}(H_0).\end{equation*}
	\noindent
	\emph{Step 4.} Recall that the periodic Sturm--Liouville operator $H_0$ is semibounded from below. Let $\lambda<\inf \sigma_{\mathrm{ess}}(H_{0})$, that is, $\operatorname{Re} c>0$ by Lemma~\ref{pingpong}~(i). It is no restriction to assume that the solutions $u_0$ and $u_1$ provided by Lemma~\ref{pingpong}~(i) and Lemma~\ref{lachs}~(i) are real-valued. Since $H_0$ is semibounded from below
 the differential expression $\tau_0-\lambda$ is non-oscillatory (see \cite[Theorem~14.9]{Weidmann87}), that is,
  $u_0$ has at most finitely many zeros in $(a,\infty)$. Furthermore, Lemma~\ref{pingpong}~(i) implies that the function $\tilde u_0$ given by
	\begin{equation*}
		\tilde u_0(x) = \mathrm{e}^{c\frac{x-a}{\omega}} u_0(x)
	\end{equation*}
	is $\omega$-periodic. Therefore, the solution $u_0$ has no zeros and
	\begin{equation*}
		\gamma:= \inf_{t\in (a,\infty)} \vert \tilde u_0(t) \rvert = \min_{t\in [a,a+\omega]} \vert \tilde u_0(t) \rvert>0.
	\end{equation*}
	Assume that $H_1$ is not semibounded from below. Then \cite[Theorem~14.9]{Weidmann87} implies that the differential expression $\tau_1-\lambda$ is oscillatory,
	and hence the solution $u_1$ of $(\tau_1-\lambda)u=0$ has infinitely many zeros $x_0<x_1<x_2<\dots$ accumulating at $\infty$. Together with \eqref{pixel} we obtain
	\begin{equation*}
		0<\gamma\leq \vert \tilde u_0(x_n) \rvert = \lvert\mathrm{e}^{c\frac{x_n-a}{\omega}}u_0(x_n)\rvert = \mathrm{e}^{\operatorname{Re} c\frac{x_n-a}{\omega}}\,\vert u_0(x_n)-u_1(x_n) \rvert\rightarrow 0\quad\text{as }n\rightarrow\infty;
	\end{equation*}
	a contradiction. This shows the semiboundedness of $H_1$.
\end{proof}

\begin{proof}[Proof of Theorem~\ref{thm2a}]
Suppose that  \eqref{xx1b} (and hence also \eqref{xx1a}) holds. We show that every gap of the essential spectrum of $H_1$ contains at most finitely many eigenvalues of $H_1$. The proof is similar as in Step 4 in the proof of Theorem~\ref{thm1a}, but instead of the zeros of solutions we consider the zeros of modified Wronskians.
	Let $\mu$,~$\lambda\in\mathbb R$ such that $\mu<\lambda$ with $\sigma_{\mathrm{ess}}(H_0)\cap (\mu,\lambda)=\sigma_{\mathrm{ess}}(H_1)\cap (\mu,\lambda) =\emptyset$. We have
	\begin{equation*}
		\lambda,\,\mu\in \partial \sigma_{\mathrm{ess}}(H_0) \cup \big(\mathbb R\setminus\sigma_{\mathrm{ess}}(H_0)\big).
	\end{equation*}
	Let $c(\lambda)$, $c(\mu)$ be the Floquet exponents associated with $(\tau_0-\lambda)u=0$ and $(\tau_0-\mu)u=0$, respectively. For the real-valued solutions $u_j(\cdot,\lambda)$ and $u_j(\cdot,\mu)$, where $j=0$, $1$, provided by Lemma~\ref{pingpong}~(i),~(iii) and Lemma~\ref{lachs}~(i),~(iii) we consider the modified Wronskians
	\begin{equation*}
		W_j(x) :=  W(u_j(\cdot,\mu),u_j(\cdot,\lambda))(x) = \begin{pmatrix}u_j(x,\lambda)\\ p_j(x)u_j'(x,\lambda)\end{pmatrix}^\top \begin{pmatrix}
			0 & -1 \\ 1 & 0
		\end{pmatrix}
		\begin{pmatrix}u_j(x,\mu) \\ p_j(x)u_j'(x,\mu)\end{pmatrix}
	\end{equation*}
	Observe that
	\begin{equation}
		\label{feynmann}
		\widetilde W_0(x) := \exp\left(\bigl(c(\lambda)+c(\mu)\bigr)\frac{x-a}{\omega}\right) W_0(x) = \bigl(U_0(x,\lambda)\bigr)^\top \begin{pmatrix}
			0 & -1 \\ 1 & 0
		\end{pmatrix} U_0(x,\mu),
	\end{equation}
	where $U_0(\cdot,\lambda)$ and $U_0(\cdot,\mu)$ are $\omega$-periodic functions given by \eqref{quasiperSol1} in Lemma~\ref{pingpong}. Therefore, the function
	$\widetilde W_0$ is $\omega$-periodic. Since there is at most one simple eigenvalue of $H_0$ in $(\mu,\lambda)$ we conclude from
	\cite[Theorem~7.5~(i)]{GesztesySimonTeschl96}
	that  $W_0$ has at most finitely many zeros in $(a,\infty)$. According to the periodicity of $\widetilde W_0$ together with \eqref{feynmann}, the modified
	Wronskian $W_0$ has no zeros and
	\begin{equation*}
		\gamma:= \inf_{t\in (a,\infty)} \vert \widetilde W_0(t) \rvert = \min_{t\in [a,a+\omega]} \vert \widetilde W_0(t) \rvert>0.
	\end{equation*}
	The difference of $W_0$ and $W_1$ can be written as
	\begin{equation*}
		\begin{split}
		W_0(x)-W_1(x) ={ }& \left(\begin{pmatrix}u_0(x,\lambda)\\ (p_0(x)u_0'(x,\lambda)\end{pmatrix} - \begin{pmatrix}u_1(x,\lambda)\\ p_1(x)u_1'(x,\lambda)\end{pmatrix}\right)^\top
		\begin{pmatrix}
			0 & -1 \\ 1 & 0
		\end{pmatrix}
		\begin{pmatrix}u_0(x,\mu) \\ p_0(x)u_0'(x,\mu)\end{pmatrix} \\[0.5\baselineskip]
		&{ }+
		\begin{pmatrix}u_1(x,\lambda)\\ p_1(x)u_1'(x,\lambda)\end{pmatrix}^\top \begin{pmatrix}
			0 & -1 \\ 1 & 0
		\end{pmatrix}
		\left(\begin{pmatrix}u_0(x,\mu) \\ p_0(x)u_0'(x,\mu)\end{pmatrix} - \begin{pmatrix}u_1(x,\mu) \\ p_1(x)u_1'(x,\mu)\end{pmatrix}\right).
	\end{split}
	\end{equation*}
	Combining this with Lemma \ref{pingpong}~(i),~(iii) and Lemma \ref{lachs}~(i),~(iii) we conclude
	\begin{equation}
		\label{gruenesfahrrad}
		\exp\left(\bigl(c(\lambda)+c(\mu)\bigr)\frac{x-a}{\omega}\right)\cdot \left(W_0(x)-W_1(x)\right)
		\rightarrow 0\quad\text{as }x\rightarrow \infty.
	\end{equation}
	
	Now assume that $H_1$ has infinitely many eigenvalues in $(\mu,\lambda)$. Then the
	modified Wronskian $W_1$ has infinitely many zeros $x_0<x_1<x_2<\dots$ which necessarily accumulate at $\infty$; cf. \cite[Theorem~7.5~(i)]{GesztesySimonTeschl96}. Then \eqref{gruenesfahrrad} implies
	\begin{equation*}
		\begin{split}
		0<\gamma \leq \vert \widetilde W_0(x_n) \rvert &= \lvert\exp\left(\bigl(c(\lambda)+c(\mu)\bigr)\frac{x_n-a}{\omega}\right) W_0(x_n)\rvert\\[0.5\baselineskip]
		&= \lvert\exp\left(\bigl(c(\lambda)+c(\mu)\bigr)\frac{x_n-a}{\omega}\right) \bigl(W_0(x_n)-W_1(x_n)\bigr)\rvert\rightarrow 0\quad \text{as }n\rightarrow \infty;
		\end{split}
	\end{equation*}
	a contradiction. Hence, $\dim\operatorname{ran} (P_{(\mu,\lambda)}(H_1))<\infty$.
 \end{proof}

\begin{proof}[Proof of Theorem~\ref{thm3a}]
 Suppose that \eqref{xx1c} (and hence also \eqref{xx1a} and \eqref{xx1b}) holds.
  We show that the boundary points of the essential spectrum of $H_1$ are no eigenvalues of $H_1$ and, therefore, $\sigma_{\mathrm{ess}}(H_1)$ is purely absolutely continuous. Let $\lambda\in \partial \sigma_{\mathrm{ess}}(H_1)$ and consider an arbitrary non-trivial linear combination $w_1:=\alpha u_1 + \beta v_1$, where $\alpha,\beta\in\mathbb C$. For the same coefficients $\alpha,\beta$ let $w_0:=\alpha u_0 + \beta v_0$ and observe that by Lemma~\ref{lachs}~(iii)
  \begin{equation}\label{www}
   w_1(x)-w_0(x)\rightarrow 0 \quad\text{and hence}\quad \lvert w_1(x)\rvert^2- \lvert w_0(x)\rvert^2\rightarrow 0\quad \text{as }x\rightarrow \infty.
  \end{equation}
 We estimate with \eqref{mint} and the boundedness of $|u_0|$ from Lemma~\ref{pingpong}~(iii) for some $M>0$ and all $t\in[a,\infty)$
  \begin{equation*}
 |w_0(t)|^2\leq\left( |\alpha| |u_0(t)| + |\beta|  C \left(1+\frac{t-a}{\omega}\right) \right)^2\leq M(1+ t^2),
 \end{equation*}
 and hence
 \begin{equation}\label{MiGente2}
 \lvert w_0(t)\rvert^2 \lvert r_1(t)-r_0(t)\rvert\leq M(1+ t^2) \lvert r_1(t)-r_0(t)\rvert.
 \end{equation}
 Moreover,
 \begin{equation}\label{takkitakki2}
	\left| \lvert w_1\rvert^2 r_1- \lvert w_0\rvert^2 r_0\right|
  \leq  \left|  \lvert w_1\rvert^2- \lvert w_0\rvert^2\right| \lvert r_1-r_0\rvert +
  \lvert w_0\rvert^2 \lvert r_1-r_0\rvert + \left|  \lvert w_1\rvert^2- \lvert w_0\rvert^2\right| r_0\\
	\end{equation}
holds pointwise a.\,e.\ on $(a,\infty)$
and by \eqref{xx1c} the functions $t\mapsto t^2\lvert r_1(t)-r_0(t)\rvert$
and $t\mapsto \lvert r_1(t)-r_0(t)\rvert$ are in $L^1(a,\infty)$.
        Thus,
\eqref{takkitakki2} together with \eqref{www}, \eqref{MiGente2}, and the
periodicity of $r_0$ imply the existence of $n_0\in\mathbb N$ such that for all $n\geq n_0$
        \begin{align*}
                \left\lvert\int_{a+n\omega}^{a+(n+1)\omega} \lvert w_1(t)\rvert^2
r_1(t)\,\mathrm dt - \int_{a+n\omega}^{a+(n+1)\omega} \lvert
w_0(t)\rvert^2 r_0(t)\,\mathrm dt\right\rvert\leq \frac{E}{2},
        \end{align*}
      where the constant $E$ is from \eqref{blin}.   This gives for all $n\geq n_0$
	\begin{equation*}
 \int_{a+n\omega}^{a+(n+1)\omega} \lvert w_1(t)\rvert^2 r_1(t)\,\mathrm dt\geq \frac{E}{2}.
	\end{equation*}
	  Therefore, $w_1$ does not belong to $L^2((a,\infty);r_1)$, which shows that $\lambda\in \partial \sigma_{\mathrm{ess}}(H_1)$ is not an eigenvalue of $H_1$.
\end{proof}

	  \begin{proof}[Proof of Proposition~\ref{propi}]
 Suppose that \eqref{xx1b} (and hence also \eqref{xx1a}) holds and that $r_1$ satisfies $C_0 \le r_1(t) \le C_1$ for $t$ in some neighbourhood of $\infty$ for some positive constants $C_0$, $C_1$.
 Let $\lambda$ be a boundary point of $\sigma_{\mathrm{ess}}(H_1)$, let $u_1=u_1(\cdot,\lambda)$ be the solution found in Lemma~\ref{lachs} (iii), and
 suppose $v_1=v_1(\cdot,\lambda)$ were an eigenfunction. Then, by \eqref{pixel} and
 \eqref{blin}, $u_1$ and $v_1$ must be linearly independent and we can rescale $v_1$ such that
 the Wronskian with $u_1$ satisfies
 \begin{equation*}
 1 = W(u_1, v_1) = u_1 (p_1 v_1') -( p_1 u_1') v_1.
 \end{equation*}
In particular, we obtain
 \begin{equation*}
 \frac{1}{2} \leq r_1 u_1^2 \frac{(p_1 v_1')^2}{r_1} + r_1 v_1^2 \frac{(p_1 u_1')^2}{r_1}
 \end{equation*}
 Now since $v_1$ is an eigenfunction, we have $r_1 v_1^2 \to 0$ (at least for some subsequence).
 Moreover, by \eqref{pixel} and our assumption on $r_1$ both $r_1 u_1^2$ and $(p_1 u_1')^2/r_1$ are bounded.
 Finally, the assumption $\int_{t-1}^{t+1} |q_1(s)|^2 ds \le C_1$ together with the other assumptions on $r_1$ and $p_1$ ensure that
 the first integral on the right hand side of \cite[Eq. (2.21) in Lemma 2.7]{sst} is bounded and hence this lemma implies
 $(p_1 v_1')^2/r_1 \to 0$, which gives a contradiction. Thus, there is no square summable solution for $\lambda$.
\end{proof}

\section{Proof of the main results}

\begin{proof}[Proofs of  Theorem~\ref{thm1}--\ref{thm3}]
Our main results follow from
a coupling argument and applications of Theorems~\ref{thm1a}, \ref{thm2a} and~\ref{thm3a} and their counterparts on the half-line $(-\infty,a)$.
More precisely,
choose any self-adjoint realization  $A_{0,-}$ and $A_{0,+}$ of $\tau_0$
in $L^2((-\infty,a);r_0)$ and $L^2((a,\infty);r_0)$, respectively, and observe that the resolvent difference
of $A_{0}$ and $A_{0,-}\oplus A_{0,+}$ is an operator of rank one or rank two.
In particular, $A_{0}$ and $A_{0,-}\oplus A_{0,+}$ have the same essential spectrum, and the periodicity also
implies $\sigma_{\mathrm{ess}}(A_{0,-})=\sigma_{\mathrm{ess}}(A_{0,+})$.

Let $A_{1,-}$ and $A_{1,+}$ be self-adjoint realizations of $\tau_1$ in $L^2((-\infty,a);r_1)$ and $L^2((a,\infty);r_1)$, respectively.
It follows from  Theorem~\ref{thm1a} that $A_{1,\pm}$ are semibounded,
$\sigma_{\mathrm{ess}}(A_{0,\pm})=\sigma_{\mathrm{ess}}(A_{1,\pm})$, and hence $A_{1,-}\oplus A_{1,+}$ is semibounded and
 \begin{equation*}
\sigma_{\mathrm{ess}}(A_0) = \sigma_{\mathrm{ess}}(A_{0,-}\oplus A_{0,+})=
\sigma_{\mathrm{ess}}(A_{1,-}\oplus A_{1,+}).
\end{equation*}
As also the resolvent difference
of $A_{1}$ and $A_{1,-}\oplus A_{1,+}$ is an operator of rank one or rank two we conclude that
$A_1$ is semibounded and
\begin{equation*}
\sigma_{\mathrm{ess}}(A_0) =\sigma_{\mathrm{ess}}(A_1).
\end{equation*}
In order to prove  Theorem~\ref{thm1} it remains to show the statement on the
 absolutely continuous  spectrum of $A_1$. Let $\lambda$ be an interior point of $\sigma_{\mathrm{ess}}(A_1)$
 and let $u$ be a non-trivial solution of $(\tau_1-\lambda)u=0$. Step 2 of the proof of Theorem~\ref{thm1a} shows that the restrictions of $u$ onto $(-\infty,a)$ and $(a,\infty)$
 are not sequentially subordinant at $\pm \infty$ and from \cite[Theorem 2]{G05} we conclude that the
spectrum of $A_1$ is purely absolutely continuous
in the interior of the spectral
bands. This completes the proof of Theorem~\ref{thm1}. Arguing with the restrictions of $u$ onto $(-\infty,a)$ and $(a,\infty)$
in the same way as in the proof of Theorem~\ref{thm3a} we also conclude that the band edges are no eigenvalues of $A_1$ under the assumptions of Theorem~\ref{thm3}.
To conclude Theorem~\ref{thm2} note that by Theorem~\ref{thm2a} each gap
contains at most finitely many
eigenvalues of $A_{1,-}\oplus A_{1,+}$.
 As
the resolvent difference
of $A_{1}$ and $A_{1,-}\oplus A_{1,+}$ is at most of rank two
the number of eigenvalues of $A_1$ in each gap can increase
by at most two, which shows Theorem~\ref{thm2}.
\end{proof}


\begin{thebibliography}{XXXX}

\bibitem{BrownEasthamSchmidt13} B.M. Brown, M.S.P. Eastham and K.M.\ Schmidt, \textit{Periodic Differential Operators}, Birkh\"auser, Basel, 2013.

\bibitem{F1} N.E. Firsova, An inverse scattering problem for the perturbed Hill operator,  Mat. Zametki 18 (1975), 831--843.

\bibitem{F2} N.E. Firsova, The direct and inverse scattering problems for the one-dimensional perturbed Hill operator, Math.\ USSR Sb. 58 (1987), 351--388.

\bibitem{F3} N.E. Firsova, Resonances of the perturbed Hill operator with exponentially decreasing extrinsic potential, Mat. Zametki 36 (1984), 711--724.

\bibitem{GesztesySimonTeschl96} F.\ Gesztesy, B.\ Simon and G.\ Teschl, Zeros of the Wronskian and renormalized oscillation theory, Am. J. Math. 118 (1996), 571--594.

\bibitem{G05} D. Gilbert, Asymptotic methods in the spectral analysis
of Sturm--Liouville operators, In: W.O. Amrein, A.M. Hinz, and D.B. Pearson (eds.) Sturm--Liouville Theory: Past and Present, Birkh\"auser Basel (2005), 121--136.

\bibitem{G13} K. Grunert, Scattering theory for Schr\"odinger operators on steplike, almost periodic infinite-gap backgrounds, J. Diff. Eq. 254 (2013), 2556--2586.

\bibitem{kt3} H.\ Kr\"uger and G.\ Teschl, \textit{Effective Pr\"ufer angles and relative oscillation criteria}, J.\ Differential Equations \textbf{245} (2008), 3823--3848.

\bibitem{K16} P. Kuchment, An overview of periodic elliptic operators, Bull. Amer. Math. Soc. 53 (2016), 343--414.

\bibitem{ROFE-BEKETOV} F.S. Rofe-Beketov, A test for the finiteness of the number of discrete levels introduced into the gaps of a
continuous spectrum by perturbations of a periodic potential, Dokl. Akad. Nauk SSSR 156 (1964), 515--518.

\bibitem{S00} K.M. Schmidt, Critical coupling constants and eigenvalue asymptotics
of perturbed periodic Sturm--Liouville operators, Comm. Math. Phys. 211 (2000), 465--485.

\bibitem{sst} M.\ Schmied, R.\ Sims, and G.\ Teschl, On the absolutely continuous spectrum of Sturm--Liouville operators with applications to radial quantum trees, Oper.\ Matrices 2 (2008), 417--434.

\bibitem{S91} G.\ Stolz, On the absolutely continuous spectrum of perturbed periodic Sturm--Liouville operators, J. Reine Angew. Math. 416 (1991), 1--23.

\bibitem{te} G.\ Teschl, \textit{Mathematical Methods in Quantum Mechanics. With Applications to Schr\"odinger Operators}, 2nd.\ ed., Amer.\ Math.\ Soc., Providence, 2014.

\bibitem{Weidmann87} J.\ Weidmann,  \textit{Spectral Theory of Ordinary Differential Operators}, Lecture Notes in Math.\ 1258, Springer, 1987.

\bibitem{Weidmann03} J.\ Weidmann, \textit{Lineare Operatoren in Hilbertr\"aumen. Teil II: Anwendungen}, Teubner, Stuttgart, 2003.

\end{thebibliography}
\end{document}